\documentclass[eprint,12pt]{elsarticle}

\usepackage{amsmath}
\usepackage{amssymb}
\usepackage{amsthm}
\usepackage{tikz}
\usepackage{tikz-3dplot}
\usepackage{dsfont}

\newcommand{\Z}{\ensuremath{\mathbb{Z}}}
\newcommand{\Q}{\ensuremath{\mathbb{Q}}}
\newcommand{\R}{\ensuremath{\mathbb{R}}}
\newcommand{\Sm}{\ensuremath{\mathcal{S}}}
\DeclareMathOperator{\GL}{GL}
\DeclareMathOperator{\mn}{Min}
\DeclareMathOperator{\vol}{Vol}
\DeclareMathOperator{\Tr}{Tr}
\newcommand{\V}{\ensuremath{\mathcal{V}}}
\newcommand{\norm}[1]{\left\lVert#1\right\rVert}

\journal{Advances in Mathematics}

\newtheorem{theorem}{Theorem}[section]
\newtheorem{lemma}[theorem]{Lemma}
\newtheorem{cor}[theorem]{Corollary}

\theoremstyle{definition}

\numberwithin{equation}{section}

\tikzset{surface/.style={draw=black, fill=blue!40!white, fill opacity=.0}}
\tikzset{surface2/.style={fill=black!40!white, fill opacity=.3}}
\tikzset{boundary1/.style={draw=black, semithick, dashed}}
\tikzset{boundary2/.style={draw=black, semithick}}
\tikzset{boundary3/.style={draw=black, ultra thin}}
\tikzset{smalltext/.style={font=\tiny}}

\begin{document}

\begin{frontmatter}



\title{An upper bound on the number of perfect quadratic forms \footnote{Supported by the ERC Advanced Grant 740972 (ALGSTRONGCRYPTO).}}


\author{W.P.J. van Woerden}%
\ead{wvw@cwi.nl}
\address{Cryptology Group, CWI, Amsterdam, The Netherlands}

\begin{abstract}
   In a recent publication Roland Bacher showed that the number $p_d$ of non-similar perfect $d$-dimensional quadratic forms satisfies $e^{\Omega(d)} < p_d < e^{O(d^3\log(d))}$. We improve the upper bound to $e^{O(d^2\log(d))}$ by a volumetric argument based on Voronoi's first reduction theory.
\end{abstract}


\end{frontmatter}


\section{Introduction} 
\label{sec:introduction}
The sphere packing problem is a classical problem with connections to fields of mathematics, information
theory and physics. This problem asks how to pack $d$-dimensional identical balls in $\R^d$ such that their density, the proportion of $\R^d$ they fill, is maximized. All best known sphere packings up to dimension $9$ are in fact lattice packings, i.e. sphere packings such that the centers of the balls form a discrete additive group. Therefore a natural restriction of the sphere packing problem is the lattice packing problem. In 1908, in his famous work \cite{voronoi1908nouvelles}, Voronoi introduced an algorithm that solves the lattice packing problem in any dimension in finite time. Voronoi showed that any lattice with optimal packing density must correspond to a so-called perfect (quadratic) form and his algorithm enumerates the finitely many perfect forms up to similarity in a fixed dimension. However, the number of non-similar perfect forms grows super-exponentially in the dimension and as a result Voronoi's algorithm has only been completely executed up to dimension $8$ \cite{korkine1873formes,barnes1957complete,jaquet1993enumeration,sikiric2007classification}. The exact number of non-similar perfect forms from dimension $2$ up to $8$ is $1,1,2,3,7,33$ and $10916$ respectively and in dimension $9$ more than $20$ million were found \cite{vanwoerdenperfect}. An intriguing question is to characterize the growth of the number of non-similar perfect forms. 

A bound on the number of perfect forms has consequences beyond estimating the complexity of Voronoi's algorithm. In 1998 C. Soul\'e \cite{soule1999perfect} proved an upper bound of $e^{O(d^4\log(d))}$, which he used to prove a statement related to Vandiver's Conjecture.

Furthermore, from the field of physics, there is interest in the statistical analysis of variations of Voronoi's algorithm based on random walks, first introduced by A. Andreanov and A. Scardicchio \cite{andreanov2012random}. They conjecture a growth of $e^{\Theta(d^2)}$ and several variations of Voronoi's algorithm are designed \cite{andreanov2012random,andreanov2016extreme} under the assumption that this conjecture is true. 

Recently R. Bacher \cite{bacher2018number} proved a lower bound of $e^{\Omega(d)}$ and an upper bound of $e^{O(d^3\log(d))}$. Bacher already conjectured our improved upper bound of $e^{O(d^2\log(d))}$ substantiated by heuristic arguments. However his proof methods and heuristic arguments do not seem to overlap with the proof we state here. 

A useful property of $d$-dimensional perfect forms is that each of them has a corresponding full rank cone inside the cone of $d$-dimensional positive semidefinite quadratic forms. In fact if we look at all perfect forms up to scaling their corresponding cones are essentially disjoint. We use this property to prove the upper bound of $e^{O(d^2 \log(d))}$ with a volumetric argument. We show that each perfect form is similar to a perfect form of which the corresponding cone has at least a certain volume. As only a certain amount of such cones fit in a disjoint manner in the cone of positive semidefinite quadratic forms we obtain an upper bound on the number of similarity classes of perfect forms.

In Section \ref{sec:basic_definitions_and_facts} we discuss the preliminaries needed for the proof and in Section \ref{sec:on_the_number_of_perfect_lattices} we prove the upper bound of $e^{O(d^2 \log(d))}$. 


\section{Preliminaries} 
\label{sec:basic_definitions_and_facts}

\subsection{Notation} 
\label{sub:notation}

We denote the sets of integers, rationals and reals by $\Z, \Q$ and $\R$ respectively. With $\R_{\geq 0}$ and $\R_{>0}$ we denote the set of all non-negative and positive reals respectively. The set of integers $\{ 1, 2, \ldots, m\}$ is denoted by $\lbrack m \rbrack$ for any integer $m \geq 1$. A vector $v \in \R^d$ is interpreted as a single column matrix $v \in \R^{d \times 1}$. The $i$-th standard unit vector is denoted by $e_i \in \Z^d$. The transpose of a vector $v \in \R^d$ or of a matrix $A \in \R^{d \times d}$ is denoted by $v^t$ or $A^t$ respectively. The standard inner product is often denoted by $v^t w = \sum_{i=1}^d v_i \cdot w_i \in \R$ and the outer product by $vw^t = (v_i \cdot w_j)_{i,j} \in \R^{d \times d}$ for vectors $v,w \in \R^d$. The trace and determinant of a square matrix $A$ are denoted by $\Tr(A)$ and $\det(A)$ respectively. The interior of a measurable set $S \subset \R^n$, i.e. the largest open set contained in $S$, is denoted by $\text{Int}(S)$. A cone is a set $C \subset \R^d$ that is closed under positive scaling. Let $X \subset \R^d$, then we denote by $\text{cone}(X)$ the cone given by all non-negative linear combinations of the elements in $X$ and by $\text{conv}(X)$ the convex set given by all convex combinations of the elements in $X$. Furthermore we denote by $\text{rank}(X)$ the dimension of the linear subspace spanned by the elements in $X$. 


\subsection{Quadratic forms} 
\label{sub:quadratic_form_}
  We associate with every symmetric real matrix $Q \in \R^{d \times d}$ a (real) \emph{quadratic form} in $d \geq 1$ variables given by
  \begin{align*}
    Q :\  &\R^d \to \R, \\
      &x \mapsto Q\lbrack x \rbrack := x^tQx.
  \end{align*}
Remark that $Q\lbrack x \rbrack = Q \lbrack -x \rbrack$ for all $x \in \R^d$. The \emph{space of all quadratic forms} is denoted by
\begin{align*}
  \Sm^d := \{ Q \in \R^{d \times d} : Q^t = Q \}.
\end{align*}
Note that $\Sm^d$ is an $n:=\binom{d+1}{2}$-dimensional real vector space, which is a Euclidean space when endowed with the standard \emph{trace inner product}
\begin{align*}
  \langle P, Q \rangle := \Tr(P^tQ) = \sum_{i,j \in \lbrack d \rbrack} P_{ij}Q_{ij}.
 \end{align*}
The norm induced by this inner product is the standard Frobenius norm.
By cyclicity of the trace, we have $x^tQx = \langle Q, xx^t \rangle$. Under any fixed ordering of the indices $\{ (i,j) \in \lbrack d \rbrack \times \lbrack d \rbrack : i \leq j \}$ a natural isometry $\phi$ from $\Sm^d$ to the canonical Euclidean space $\R^n$ is given by:
\begin{align*}
  \phi : \Sm^d \to \R^n : Q \mapsto ( q_{ij} )_{i \leq j} 
\end{align*}
where $q_{ii} := Q_{ii}$ and $q_{ij} := \sqrt{2}Q_{ij} = \sqrt{2}Q_{ji}$ for $i<j$. Indeed we have
\begin{align*}
  \langle \phi(P), \phi(Q) \rangle = \langle P,Q \rangle.
\end{align*}
This isometry will implicitly be used in the figures. Moreover, we consider the \emph{cone of positive definite quadratic forms} (PQFs)
\begin{align*}
  \Sm_{>0}^d := \{ Q \in \Sm^d : Q \text{ is positive definite}  \},
\end{align*}
its closure, the \emph{cone of positive semidefinite quadratic forms}
\begin{align*}
  \Sm_{\geq 0}^d := \{ Q \in \Sm^d : Q \text{ is positive semidefinite}\},
\end{align*}
and finally its historically named \emph{rational closure}  \cite{namikawa2006toroidal}
\begin{align*}
\tilde{S}_{\geq 0}^d := \text{cone}(\{ xx^t : x \in \Z^n \}) \subset \Sm_{\geq 0}^d. 
\end{align*}

\subsection{Arithmetical equivalence} 
\label{sub:arithmetical_equivalence}

Two quadratic forms are \emph{arithmetically equivalent} if they lie in the same orbit under the action $(Q, U) \mapsto U^tQU$ of the multiplicative group
\begin{align*}
  \GL_d(\Z) := \{ U \in \Z^{d \times d} : |\det U | = 1 \}
\end{align*}
of unimodular matrices. We call two PQFs $Q,Q' \in \Sm_{>0}^d$ \emph{similar} if and only if $Q$ is arithmetically equivalent to $\alpha Q'$ for some $\alpha  \in \R_{>0}$.


\subsection{Positive definite quadratic forms} 
\label{sub:positive_definite_quadratic_forms_}

For any PQF $Q \in \Sm_{>0}^d$ there exists a smallest real number $r > 0$ for which $Q \lbrack x \rbrack = r$ has an integral solution. We define this number as the \emph{arithmetical minimum} denoted by
\begin{align*}
  \lambda_1(Q) := \min\limits_{x \in \Z^d \setminus \{ 0 \}} Q\lbrack x \rbrack.
\end{align*}
More generally, we define for $i \in \lbrack d \rbrack$ the \emph{$i$-th successive minima} $\lambda_i(Q)$ as
\begin{align*}
  \lambda_i(Q) := \text{inf}\{ \lambda > 0 &: \exists \text{ $\R$-linearly independent }x_1, \ldots, x_i \in \Z^n \setminus \{ 0 \} \\
  &: Q\lbrack x_j\rbrack \leq \lambda \ \text{ for all } j \in \lbrack i \rbrack \},
\end{align*}
where the infinum is in fact a minimum. Note that $\lambda_i(\alpha Q) = \alpha \lambda_i(Q)$ for any $\alpha \in \R_{>0}$ . Furthermore, the successive minima are invariant under arithmetical equivalence, because $(U^tQU)\lbrack x \rbrack = Q\lbrack Ux \rbrack$ and $U\Z^d=\Z^d$ for all $U \in \GL_d(\Z)$. So under the assumption that $\lambda_1(Q)=\lambda_1(Q')$ for two PQFs $Q,Q'\in \Sm_{>0}^d$, the notions of similarity and arithmetical equivalence coincide. 

By applying Hermite-Korkine-Zolotarev lattice reduction to a PQF we can always find an arithmetically equivalent PQF for which the successive minima are attained, up to a factor linear in the dimension, by the standard basis of $\Z^d$.
\begin{lemma}[Lagarias, Lenstra, Schnorr \cite{Lagarias1990}] \label{hkz}
Consider a PQF $Q \in \Sm_{> 0}^d$. Then there exists a PQF $Q' \in \Sm_{> 0}^d$ arithmetically equivalent to $Q$ such that
  \begin{align*}
    Q'\lbrack e_i \rbrack = Q'_{ii} \leq \frac{i+3}{4} \lambda_i(Q) & & \text{for all } i \in \lbrack d \rbrack.
  \end{align*}\vspace{-0.6cm}
\end{lemma}
We define the set of \emph{minimal vectors} of a PQF $Q \in \Sm_{>0}^d$ as
\begin{align*}
  \mn Q := \{ x \in \Z^d : Q\lbrack x \rbrack = \lambda_1(Q) \}.
\end{align*}
Note that if $Q' = U^tQU$, then $\mn Q = U\cdot \mn Q'$. We also define what is called the \emph{Voronoi domain} $\V(Q)$ of a PQF $Q \in \Sm_{>0}$ as
\begin{align*}
  \V(Q) := \text{cone}( \{ xx^t : x \in \mn Q \} )  \subset \tilde{S}_{\geq 0}^d. 
\end{align*}
A PQF $Q \in \Sm_{>0}^d$ is called \emph{perfect} if the set of equations
\begin{align*}
   \{ Q'\lbrack x \rbrack = \lambda_1(Q) \text{ for all } x \in \mn Q \},
\end{align*}
has the unique solution $Q' = Q$ among $Q' \in \Sm^d$.  That is, a perfect form is uniquely determined by its minimal vectors. Recall that $Q \lbrack x \rbrack= \langle Q, xx^t \rangle$, therefore a PQF $Q \in \Sm_{>0}^d$ is perfect if and only if its Voronoi domain $\V(Q)$ has full rank $n=\binom{d+1}{2}$ in $\Sm^d$. In particular any perfect form has at least $n$ minimal vectors up to sign.

\newsavebox{\smlmat}
\savebox{\smlmat}{$\left(\begin{smallmatrix}2&1\\1&2\end{smallmatrix}\right)$}

\begin{figure}[h!]
\begin{center}
  \tdplotsetmaincoords{70}{70}
\tdplotsetrotatedcoords{-90}{180}{-90}
\begin{tikzpicture}[tdplot_main_coords]
  \coordinate (O) at (0,0,0);

    \begin{scope} [canvas is xy plane at z=3]
      \draw[] circle (3);
      \end{scope}
    \draw[] (O) -- (2.04, 2.2, 3);
    \draw[] (O) -- (0, -3, 3);
   \draw[boundary2] (O) -> (3,0,3);
   \draw[boundary2] (O) -> (-3,0,3);
   \draw[boundary2] (O) -> (0,3,3);
   \draw[boundary1] (3,0,3) -- (-3,0,3) -- (0,3,3) -- cycle;
   \fill[surface2] (O) -- (3,0,3) -- (0,3,3) -- cycle;
   \fill[surface2] (O) -- (-3,0,3) -- (0,3,3) -- cycle;
   \fill[surface2] (O) -- (3,0,3) -- (-3,0,3) -- cycle;
\end{tikzpicture}
  \caption{Voronoi domain of the perfect form \usebox{\smlmat} in the cone $\Sm_{\geq 0}^2$.}
\end{center}
\end{figure}

For a PQF $Q \in \Sm_{>0}^d$ we define the \emph{dual} PQF as the inverse matrix $Q^{-1} \in \Sm_{>0}^d$; this coincides with lattice duality. Note that if the PQFs $Q,Q' \in \Sm_{>0}^d$ are arithmetically equivalent by $U$, then $Q^{-1}$ and $(Q')^{-1}$ are arithmetically equivalent by $U^{-t}$. There are several metric relations between the PQFs $Q$ and $Q^{-1}$, known as \emph{transference theorems}. In particular for the successive minima we have bounds from Banaszczyk.
\begin{theorem}[Banaszczyk {\cite[Thm. 2.1]{Banaszczyk1993}}] \label{transference}
Consider a PQF $Q \in \Sm_{>0}^d$. Then the successive minima of $Q$ and its dual $Q^{-1} \in \Sm_{>0}^d$ satisfy
\begin{align*}
  \lambda_i(Q) \cdot \lambda_{d-i+1}(Q^{-1}) \leq d^2 & & \text{for all } i \in \lbrack d \rbrack.
\end{align*}
\end{theorem}





\subsection{Volume} 
\label{sub:volume_and_relative_interior}

By making use of the isometry $\phi : \Sm^d \to \R^n$ we only need the standard notion of volume in $\R^n$. The $n$-dimensional volume of a measurable set $S \subset \R^{n'}$ of affine dimension at most $n$ is denoted by $\vol_n(S)$. In particular the $n$-dimensional unit ball $B^n$ has volume
\begin{align*}
	\vol_n(B^n) = \frac{\pi^{n/2}}{\Gamma(n/2+1)},
\end{align*}
where $\Gamma$ denotes Euler's gamma function. Furthermore, a simplex, a convex set spanned by $0$ and $n$ linearly independent points $x_1, \ldots, x_n \in \R^n$, has volume
\begin{align*}
	\vol_n(\text{conv}\{ 0, x_1, \ldots, x_n \}) = \frac{1}{n!} |\det( (x_i)_{i \in \lbrack n \rbrack} )|.
\end{align*}
See \cite{simplexvolume} for a proof. If $S \subset \R^n$ is a convex measurable set of affine dimension $n-1$ and $p \in \R^n$ is a point with orthogonal distance $h$ to $S$, then it holds that
\begin{align*}
  \vol_n(\text{conv}(S \cup \{ p \})) = \frac{h}{n} \cdot \vol_{n-1}(S).
\end{align*}

We call two measurable sets $S_1, S_2 \subset \R^n$ \emph{essentially disjoint} if $\text{Int}(S_1) \cap \text{Int}(S_2) = \emptyset$. In particular for $N$ pairwise essentially disjoint simplices $S_1, \ldots, S_N \subset \R^n$ it holds that
\begin{align*}
  \vol_n\left( \bigcup_{i=1}^N S_i \right) = \sum_{i=1}^N \vol_n(S_i).
\end{align*}



\section{An upper bound on the number of perfect forms} 
\label{sec:on_the_number_of_perfect_lattices}
In this section we prove an upper bound on the number of non-similar $d$-dimensional perfect forms. The bound of $e^{O(d^{2}\log(d))}$ improves on the bound of $e^{O(d^{3}\log(d))}$ proven by R. Bacher \cite{bacher2018number}. Bacher already conjectured such an upper bound with heuristic arguments. Our proof strategy does not seem to overlap with the proof or the heuristic arguments of Bacher. 

\begin{theorem} \label{numberofperfect}
The number $p_d$ of non-similar $d$-dimensional perfect quadratic forms has an upper bound of the form $e^{O(d^2 \log(d))}$. More precisely, $p_d$ satisfies
\begin{align*}
	p_d \leq \frac{(n-1)!}{\Gamma(\frac{n}{2}+\frac{1}{2})} \cdot \sqrt{ \frac{\pi^{n-1}}{2^{7n-d}} \cdot \frac{(d-1)^{n-1}}{d^{n}} } \cdot (d^3(d+7))^n, & & \text{ where } n = \binom{d+1}{2}.
\end{align*}
\end{theorem}

The proof makes use of a volumetric argument after showing the existence of a good representative for each similarity class of perfect forms. Recall that every perfect form $Q \in \Sm_{>0}^d$ has a full rank Voronoi domain $\V(Q) = \text{cone}( \{ xx^t : x \in \mn Q \} )$ in $\Sm^d$. A key point in the proof of Theorem \ref{numberofperfect} is that the Voronoi domains of the set of perfect forms up to scaling form an essentially disjoint partitioning of the rational closure $\tilde{S}_{\geq 0}^d = \text{cone}\{ xx^t : x \in \Z^n \}$.

\begin{lemma}[Voronoi \cite{voronoi1908nouvelles}] \label{voronoisubdivision}
  The Voronoi domains of the $d$-dimensional perfect forms cover $\tilde{S}_{\geq 0}^d$. Restricted to perfect forms $Q \in \Sm_{>0}^d$ with $\lambda_1(Q) = 1$, it holds that
\begin{align*}
  \tilde{\Sm}_{\geq 0}^d = \bigcup\limits_{\substack{Q \text{ perfect}\\ \lambda_1(Q)=1}} \V(Q),
\end{align*}
  where the Voronoi domains are essentially disjoint.
\end{lemma}
\begin{proof}
  This result originates from the first reduction theory of Voronoi \cite{voronoi1908nouvelles}, see section 7.1 of \cite{martinet2002perfect} for a full proof. We do reprove the last part of the Lemma, that is, the part that the union is essentially disjoint. This is the only part from this Lemma that we need. Let $Q,Q' \in \Sm_{>0}^d$ be two perfect forms where we assume that $\lambda_1(Q) = \lambda_1(Q') = 1$. Suppose that there exists an $R \in \text{Int}(\V(Q)) \cap \text{Int}(\V(Q'))$. We have to show that $Q=Q'$. Because $R \in \text{Int}(\V(Q))$, there exist positive $c_x \in \R_{>0}$ for every $x \in \mn Q$ such that $R = \sum_{x \in \mn Q} c_x \cdot x x^t$. As a result we have
  \begin{align*}
    \langle R, Q' \rangle = \sum_{x \in \mn Q} c_x \cdot x^t Q' x \geq \sum_{x \in \mn Q} c_x = \sum_{x \in \mn Q} c_x \cdot x^t Q x = \langle R, Q \rangle,
  \end{align*}
  using that $\lambda_1(Q) = \lambda_1(Q') = 1$. Because $R \in \text{Int}(\V(Q'))$ we get symmetrically the inequality $\langle R, Q' \rangle \leq \langle R,Q \rangle$ and thus equality. Then we have
  \begin{align*}
    0 = \langle R, Q'-Q \rangle = \sum_{x \in \mn Q} c_x \left( x^tQ'x - 1 \right).
  \end{align*}
  Because $c_x > 0$ and $x^tQ'x \geq 1$ for all $x \in \mn Q$, it holds that $x^tQ'x  = 1$ for all $x \in \mn Q$, i.e. $\mn Q \subset \mn Q'$. We conclude by perfectness of $Q$ that $Q' = Q$. 
\end{proof}

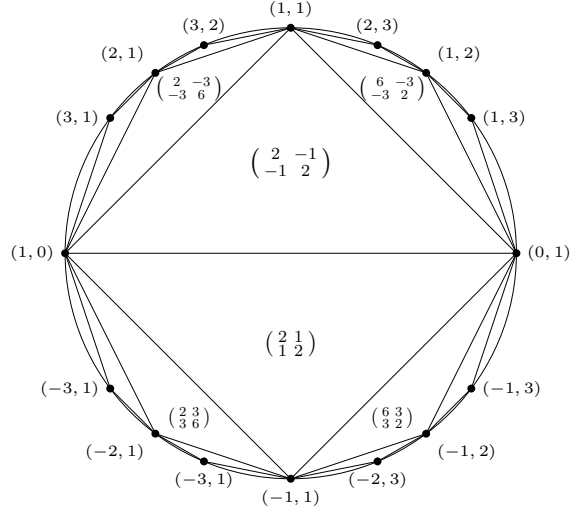
\begin{figure}[h!]


\begin{center}\begin{tikzpicture}[scale=3]
    \draw[boundary3] (0,0) circle[radius=1];
    \draw[boundary3] (-1,0) -- (1,0) -- (0,1) -- cycle;
    \draw[boundary3] (-1,0) -- (0,-1) -- (1,0);
    \fill[smalltext] (1,0) circle[radius=0.5pt] node[anchor=west] {$(0,1)$};
    \fill[smalltext] (-1,0) circle[radius=0.5pt] node[anchor=east] {$(1,0)$};
    \fill[smalltext] (0,1) circle[radius=0.5pt] node[anchor=south] {$(1,1)$};
    \fill[smalltext] (0,-1) circle[radius=0.5pt] node[anchor=north] {$(-1,1)$};
    \fill[smalltext] (0.6,0.8) circle[radius=0.5pt] node[anchor=south west] {$(1,2)$};
    \fill[smalltext] (0.6,-0.8) circle[radius=0.5pt] node[anchor=north west] {$(-1,2)$};
    \fill[smalltext] (-0.6,0.8) circle[radius=0.5pt] node[anchor=south east] {$(2,1)$};
    \fill[smalltext] (-0.6,-0.8) circle[radius=0.5pt] node[anchor=north east] {$(-2,1)$};
    \draw[boundary3] (0,1) -- (0.6,0.8) -- (1,0) -- (0.6,-0.8) -- (0,-1) -- (-0.6,-0.8) -- (-1,0) -- (-0.6,0.8) -- cycle;
    \fill[smalltext] (0.8,0.6) circle[radius=0.5pt] node[anchor=west] {$(1,3)$};
    \fill[smalltext] (0.8,-0.6) circle[radius=0.5pt] node[anchor=west] {$(-1,3)$};
    \fill[smalltext] (-0.8,0.6) circle[radius=0.5pt] node[anchor=east] {$(3,1)$};
    \fill[smalltext] (-0.8,-0.6) circle[radius=0.5pt] node[anchor=east] {$(-3,1)$};

    \fill[smalltext] (5/13,12/13) circle[radius=0.5pt] node[anchor=south] {$(2,3)$};
    \fill[smalltext] (5/13,-12/13) circle[radius=0.5pt] node[anchor=north] {$(-2,3)$};
    \fill[smalltext] (-5/13,12/13) circle[radius=0.5pt] node[anchor=south] {$(3,2)$};
    \fill[smalltext] (-5/13,-12/13) circle[radius=0.5pt] node[anchor=north] {$(-3,1)$};
    \draw[boundary3] (0,1) -- (5/13,12/13) -- (0.6,0.8) -- (0.8,0.6) -- (1,0) -- (0.8,-0.6) -- (0.6,-0.8) -- (5/13,-12/13) -- (0,-1) -- (-5/13,-12/13) -- (-0.6,-0.8) -- (-0.8,-0.6) -- (-1,0) -- (-0.8,0.6) -- (-0.6,0.8) -- (-5/13,12/13) -- cycle;

    \draw (0,0.4) node {\scalebox{.8}{$\left(\begin{smallmatrix}
      2&-1\\-1&2
    \end{smallmatrix}\right)$}};
    \draw (0,-0.4) node {\scalebox{.8}{$\left(\begin{smallmatrix}
      2&1\\1&2
    \end{smallmatrix}\right)$}};
    \draw (0.45,0.73) node {\scalebox{.65}{$\left(\begin{smallmatrix}
      6&-3\\-3&2
    \end{smallmatrix}\right)$}};
    \draw (-0.45,0.73) node {\scalebox{.65}{$\left(\begin{smallmatrix}
      2&-3\\-3&6
    \end{smallmatrix}\right)$}};
    \draw (-0.45,-0.73) node {\scalebox{.65}{$\left(\begin{smallmatrix}
      2&3\\3&6
    \end{smallmatrix}\right)$}};
    \draw (0.45,-0.73) node {\scalebox{.65}{$\left(\begin{smallmatrix}
      6&3\\3&2
    \end{smallmatrix}\right)$}}; 
  \end{tikzpicture}
  \caption{Partial partitioning of $\Sm_{\geq 0}^2$ by Voronoi domains viewed on a fixed trace plane. A vector $x^t$ indicates the extreme ray $xx^t$. A matrix $Q$ indicates the Voronoi domain $\V(Q)$.}\end{center}\end{figure}

To turn Lemma \ref{voronoisubdivision} into an upper bound on the number of non-similar perfect forms, we need to find, in each similarity class, a perfect form $Q \in \Sm_{>0}^d$ for which $\V(Q)$ is `large'. To find such a good representative we use the following lemma.
\begin{lemma} \label{smallx}
  For all PQFs $Q \in \Sm_{> 0}^d$, there exists a PQF $Q' \in \Sm_{> 0}^d$ arithmetically equivalent to $Q$ such that $x^tx \leq \frac{1}{8}d^3(d+7)$ for all $x \in \mn Q'$. 
\end{lemma}
\begin{proof}
  We can assume without loss of generality that $\lambda_1(Q)=1$. Applying Lemma \ref{hkz} to the dual PQF $Q^{-1}$ we obtain a PQF $Q' \in \Sm_{>0}^d$ arithmetically equivalent to $Q$ such that
\begin{align*}
  (Q'^{-1})_{ii} \leq \frac{i+3}{4} \lambda_i(Q^{-1})
\end{align*}
for all $i \in \lbrack d \rbrack$. Furthermore, note that $\lambda_d(Q) \geq \ldots \geq \lambda_1(Q) = 1$ and thus by Theorem \ref{transference} we have
\begin{align*}
  \lambda_i(Q^{-1}) \leq \frac{d^2}{\lambda_{d-i+1}(Q)} \leq d^2
\end{align*}
for all $i \in \lbrack d \rbrack$.
 Combining these inequalities we obtain
\begin{align*}
  \Tr(Q'^{-1}) = \sum_{i=1}^d (Q'^{-1})_{ii} \leq d^2 \cdot \sum_{i=1}^d \frac{i+3}{4} = \frac{1}{8}d^3(d+7).
\end{align*}
In particular, this gives a lower bound on the eigenvalues $\mu_1, \ldots, \mu_d > 0$ of $Q'$, namely
\begin{align*}
  \frac{1}{\mu_i} \leq \sum_{j=1}^d \frac{1}{\mu_j} = \Tr(Q'^{-1}) \leq \frac{1}{8}d^3(d+7).
\end{align*}
But as $\min\limits_i \mu_i = \min\limits_{y \in \R^d - 0} \frac{y^tQ'y}{y^ty}$ and $\lambda_1(Q') = \lambda_1(Q)=1$ we have
\begin{align*}
  x^t x \leq \frac{x^tQ'x}{\min\limits_i \mu_i } \leq 1 \cdot \frac{1}{8} d^3 (d+7)
\end{align*}
for all $x \in \mn Q'$.
\end{proof}

To quantify the volume of the cones in $\Sm_{\geq 0}^d$, we bound them by the half-space $T_d = \{ Q \in \Sm^d : \langle Q, I_d \rangle = \Tr(Q) \leq 1 \}$ in $\Sm^d$. Recall the isometry $\phi : \Sm^d \to \R^n$ from Section \ref{sub:quadratic_form_}. 
By Lemma \ref{smallx}, we can obtain for any similarity class a perfect form for which the Voronoi domain is reasonably large.

\begin{cor} \label{lowerbound}
	Consider a perfect quadratic form $Q \in \Sm_{>0}^d$. Then there exists a perfect form $Q' \in \Sm_{>0}^d$ arithmetically equivalent to $Q$ such that
	\begin{align*}
  \vol_n(\phi(\V(Q') \cap T_d)) \geq \frac{1}{n!} \frac{2^{(n-d)/2}}{(\frac{1}{8} d^3 (d+7))^n} =: \ell_d.
\end{align*}
\end{cor}
\begin{proof}
According to Lemma \ref{smallx}, there exists a PQF $Q' \in \Sm_{>0}^d$ arithmetically equivalent to $Q$ that satisfies $x^tx \leq \frac{1}{8}d^3(d+7)$ for all $x \in \mn Q'$.
The polytope $V' = \phi(\V(Q') \cap T_d)$ is the convex hull of $0$ and $\phi\left(\frac{xx^t}{x^tx}\right)$ for all $x \in \mn Q' / \{ \pm 1 \}$.
As we are only in search of a lower bound for the volume of $V'$, we consider without loss of generality a subset $M_{Q'} \subset \mn Q'$ such that $|M_{Q'}| = n$ and $\text{rank}\{ \phi(xx^t) \in \R^n : x \in M_{Q'} \} = n$. Note that this is possible exactly because $Q'$ is perfect. Then $V'$ contains a simplex induced by $M_{Q'}$ and we get
\begin{align*}
 \vol_n(V') \geq \vol_n\left(\text{conv}\left(\{ 0 \} \cup \left\{ \phi\left(\frac{xx^t}{x^t x}\right) : x \in M_{Q'} \right\}\right)\right) = \frac{1}{n!} |\det(W)|
\end{align*}
with
\begin{align*}
  W = \left( \phi\left(\frac{xx^t}{x^t x}\right)\right)_{x \in M_{Q'}} \in \R^{n\times n}.
\end{align*} 
By using that $\phi(xx^t) \in \Z^d \oplus \sqrt{2}\Z^{n-d}$ for all $x \in \Z^d$ under some fixed ordering and that the determinant of $W$ is nonzero, because it has full rank, we get
\begin{align*}
  |\det(W)| &= \left(\prod_{x \in M_{Q'}} \frac{1}{x^tx}\right) \cdot |\det( ( \phi(xx^t) )_{x \in M_{Q'}} )| \\
  &\geq \left(\prod_{x \in M_{Q'}} \frac{1}{x^tx} \right) \cdot 2^{(n-d)/2} \geq \frac{2^{(n-d)/2}}{(\frac{1}{8} d^3 (d+7))^n}.
\end{align*}
So we can conclude that 
\begin{align*}
  \vol_n(\phi(\V(Q') \cap T_d)) \geq \frac{1}{n!} \cdot \frac{2^{(n-d)/2}}{(\frac{1}{8} d^3 (d+7))^n}.
\end{align*}
\end{proof}

Now we have found a good representative for each similarity class of perfect forms; the upper bound quickly follows.

\begin{proof}[Proof of Theorem \ref{numberofperfect}]
Let $P_d$ be a complete set of non-equivalent representatives of perfect $d$-dimensional quadratic forms with $\lambda_1(Q) = 1$ up to arithmetical equivalence. By Corollary \ref{lowerbound} we can assume that $\vol_n(\phi(\V(Q) \cap T_d)) \geq \ell_d$ for all $Q \in P_d$. By Lemma $1$ we have
\begin{align*}
  \bigcup\limits_{Q \in P_d} \V(Q) \subset \tilde{\Sm}_{\geq 0}^d \subset \Sm_{\geq 0}^d,
\end{align*}
where the $\V(Q)$ are essentially disjoint. This yields
\begin{align*}
 |P_d| \cdot \ell_d \leq \sum\limits_{Q \in P_d} \vol_n( \phi( \V(Q) \cap T_d )) \leq\vol_n(\phi(\Sm_{\geq 0}^d \cap T_d)).
\end{align*}
What remains is to find an upper bound for the volume of $\Sm_{\geq 0}^d \cap T_d$. Remark that $\Sm_{\geq 0}^d \cap T_d$ is the convex hull of $0 \in \Sm^d$ and the convex base $C_d = \Sm_{\geq 0}^d \cap \{ Q \in \Sm^d : \Tr(Q)=1 \}$. Furthermore $0$ has orthogonal distance $\norm{ \frac{1}{d}I_d }_F = \frac{1}{\sqrt{d}}$ to $C_d$. 
For $A = ( a_{ij} )_{i,j} \in C_d$ we have $a_{ij}^2 \leq a_{ii}a_{jj}$ for all $i,j \in \lbrack d \rbrack$, because $A$ is positive semidefinite. Therefore, for any $A \in C_d$ we have
\begin{align*}
	\langle A, A \rangle = \sum_{i,j \in \lbrack d \rbrack} a_{ij}^2 \leq \sum_{i,j \in \lbrack d \rbrack} a_{ii}a_{jj} = \Tr(A)^2 = 1.
\end{align*}
But then 
\begin{align*}
	\left\langle A-\frac{1}{d}I_d, A-\frac{1}{d}I_d \right\rangle = \langle A,A \rangle - \frac{2}{d} \Tr(A) + \frac{1}{d} \leq \frac{d-1}{d},
\end{align*}
and thus $C_d$ is contained in an $(n-1)$-dimensional ball with center $\frac{1}{d}I_d$ and radius $\sqrt{\frac{d-1}{d}}$. This implies the following upper bound.
\begin{align*}
	\vol_n(\phi(\Sm_{\geq 0}^d \cap T_d)) &\leq \vol_n(\phi(\text{conv}(C_d \cup \{ 0 \}))) \\
  &= \frac{1}{n} \cdot \frac{1}{\sqrt{d}} \cdot \left( \frac{d-1}{d} \right)^{\frac{n-1}{2}} \cdot \vol_{n-1}(B^{n-1}) =: u_d.
\end{align*}


Recall that $n = \binom{d+1}{2} = O(d^2)$ and $n! \leq n^n = e^{O(d^2 \log(d))}$. To conclude,
\begin{align*}
 |P_d| \leq \frac{u_d}{\ell_d} &= \frac{(n-1)!}{\Gamma(\frac{n}{2}+\frac{1}{2})} \cdot \sqrt{ \frac{\pi^{n-1}}{2^{7n-d}} \cdot \frac{(d-1)^{n-1}}{d^{n}} } \cdot (d^3(d+7))^n = e^{O(d^2\log d)}.
\end{align*}

\end{proof}

\section{An upper bound on the arithmetical minimum of perfect forms} 
\label{sec:an_upper_bound_on_the_arithmetical_minimum_of_perfect_forms}
We are grateful for the anonymous reviewer that suggested an additional remark, namely that Lemma \ref{smallx} also results in an explicit upper bound on the arithmetical minimum $\lambda_1(Q)$ of a primitive integral perfect form $Q$. We denote the lattice of integral symmetric matrices by $\Sm^d \left( \Z \right) \subset \Sm^d$.

\begin{theorem}
  Let $Q \in \Sm_{>0}^d\left( \Z \right)$ be a primitive and integral perfect form, then the arithmetical minimum $\lambda_1(Q)$ satisfies
  \begin{align*}
    \lambda_1(Q) \leq 2^{-(n+d/2)} \cdot \left( d^3(d+7)\right)^{n/2}, \text{ where }n = \binom{d+1}{2}.
  \end{align*}
\end{theorem}


\begin{proof}
By Lemma $\ref{smallx}$ we can assume that $Q$ has $n$ minimal vectors $x_1, \ldots, x_n \in \Z^d$ with $x_i^tx_i \leq \frac{1}{8}d^3(d+7)$ that define a full rank system
  \begin{align*}
    \langle x_ix_i^t, Q' \rangle = \lambda_1(Q) \text{  for all } i \in \lbrack n \rbrack,
  \end{align*}
with the unique solution $Q'=Q$. This translates to an integral linear system
\begin{align*}
  A \cdot (Q_{ij}')_{i \leq j} = \lambda_1(Q) \cdot \mathds{1}^n & & \text{with } A := \left( \phi'(x_ix_i^t) \right)_{i \in \lbrack n \rbrack} \in \Z^{n \times n},
\end{align*}
using the embedding
\begin{align*}
  \phi' : \Sm^d(\Z) &\to \Z^d \oplus 2\Z^{n-d} : P \mapsto (p_{ij})_{i\leq j},
\end{align*}
where $p_{ii}:=P_{ii}$ and $p_{ij} := 2P_{ij}$ for $i < j$. This allows us to express $Q$ as
\begin{align*}
  \left( Q_{ij} \right)_{i\leq j} = \frac{\lambda_1(Q)}{\det(A)} \cdot \text{adj}(A) \cdot \mathds{1}^n \in \Z^n.
\end{align*}
The adjugate $\text{adj}(A)$ of an integral matrix is integral and thus $\frac{\det(A)}{\lambda_1(Q)} \cdot Q$ is integral.
As $Q$ is primitive we get that $\lambda_1(Q) \leq \det(A)$ and we conclude by applying the Hadamard inequality as follows:
\begin{align*}
  \det(A) &= \det \left(\phi'(x_ix_i^t)\right)_{i \in \lbrack n \rbrack} = 2^{(n-d)/2} \cdot \det \left(( \phi(x_ix_i^t)\right)_{i \in \lbrack n \rbrack} \\
  &\leq 2^{(n-d)/2} \cdot \prod\limits_{i=1}^n \norm{\phi(x_ix_i^t)} \leq 2^{(n-d)/2} \cdot \left( \frac{1}{8}d^3(d+7) \right)^{n/2} \leq e^{O(d^2\log d)}.
\end{align*}
\end{proof}


\bibliographystyle{elsarticle-num} 
\bibliography{citations}

\end{document}